\documentclass[12pt,a4paper]{article}
\usepackage[active]{srcltx}
\usepackage{amsfonts}
\usepackage{amsmath}
\usepackage{amsbsy}
\usepackage{amsxtra}
\usepackage{latexsym}
\usepackage{amssymb}
\usepackage{url}
\usepackage{pstricks,pst-node}
\usepackage{enumerate}
\makeatletter
\g@addto@macro\thesection.
\makeatother

\newtheorem{theorem}{Theorem}
\newtheorem{lemma}{Lemma}

\newcommand{\lng}{\langle}
\newcommand{\rng}{\rangle}

\newcommand{\sa}{\sqcap}
\newcommand{\su}{\sqcup}
\newcommand{\R}{\mathbb R}

\newcommand{\Hi}{\mathbb H}

\newenvironment{proof}{{\noindent\bf Proof.}}{\hfill$\Box$\\}

\begin{document}

\title{A duality between the metric projection onto a convex cone and the metric projection 
onto its dual in Hilbert spaces\thanks{{\it 1991 A M S Subject Classification:} Primary 90C33,
Secondary 15A48; {\it Key words and phrases:} convex sublattices, isotone projections. }}
\author{S. Z. N\'emeth\\School of Mathematics, The University of Birmingham\\The Watson Building, Edgbaston\\Birmingham B15 2TT, United Kingdom\\email: nemeths@for.mat.bham.ac.uk}
\date{}
\maketitle

\begin{abstract}
	If $K$ and $L$ are mutually dual closed convex cones in a Hilbert space $\Hi$ with the 
	metric projections onto them denoted by $P_K$ and $P_L$ respectively, then the 
	following two assertions are equivalent: (i) $P_K$ is isotone with  respect to the 
	order induced by $K$ (i. e. $v-u\in K$ implies $P_Kv-P_Ku\in K$);  (ii) $P_L$ is 
	subadditive with respect to the order induced by $L$ (i. e. $P_Lu+P_Lv-P_L(u+v)\in L$ 
	for any $u,\;v \in \Hi$). This extends the similar result of A. B. N\'emeth and 
	the author for Euclidean spaces. The extension is essential because the proof of 
	the result for Euclidean spaces is essentially finite dimensional and seemingly 
	cannot be extended for Hilbert spaces. The proof of the result for Hilbert spaces is 
	based on a completely different idea which uses extended lattice operations.
\end{abstract}

\section{Introduction}
For simplicity let us call a closed convex cone simply cone.
Both the isotonicity \cite{IsacNemeth1990b,IsacNemeth1990c} and the subadditivity 
\cite{AbbasNemeth2012,NemethNemeth2012}, of a projection onto a pointed cone with respect to 
the order defined by the cone can be used for iterative methods for finding solutions of 
complementarity problems with respect to the cone. Iterative methods are widely used for 
solving various types of equilibrium problems (such as variational inequalities, 
complementarity problems etc.) In the recent years the isotonicity gained more and more ground
for handling such problems (see \cite{NishimuraOk2012}, \cite{CarlHeikilla2011} and the 
large number of references in \cite{CarlHeikilla2011} related to ordered vector spaces). If
a complementarity problem is defined by a cone $K\subset\Hi$ 
and a mapping $f:K\to\Hi$, where $(\Hi,\lng\cdot,\cdot\rng)$ is a Hilbert space, then $x$ 
is a solution of the corresponding complementarity problem (that is, $x\in K$, $f(x)\in K^*$ 
and $\lng x,f(x)\rng=0$, where $K^*$ is the dual of $K$), if and only if $x=P_K(x-f(x))$. 
Thus, if
$f$ is continuous and the sequence $x^n$ given by the iteration $x^{n+1}=P_K(x^n-f(x^n))$ is 
convergent, then its limit is a fixed point of the mapping $x\mapsto P_K(x-f(x))$ and 
therefore
a solution of the corresponding complementarity problem. A specific way for showing the 
convergence of the sequence $x^{n+1}=P_K(x^n-f(x^n))$ is to use the isotonicity 
\cite{IsacNemeth1990b,IsacNemeth1990c} or subadditivity \cite{AbbasNemeth2012,NemethNemeth2012}
of $P_K$ with respect to the order induced by the cone $K$. In finite dimension the 
isotonicity (subadditivity) of $P_K$ imposes strong constraints on the structure of $K$ 
($K^*$). If $P_K$ is isotone (subadditive), then $K$ ($K^*$) has to be a direct sum of the 
subspace $V=K\cap(-K)$ ($V=K^*\cap(-K^*)$) with a latticial cone of a specific structure in 
the orthogonal complement of the subspace $V$ (see 
\cite{GuyaderJegouNemeth2012,IsacNemeth1992,NemethNemeth2012}). There exist cones of this 
type which are important from the practical point of view, such as the monotone 
cone (see \cite{GuyaderJegouNemeth2012}) and the monotone nonnegative cone 
(see \cite{Dattorro2005}). For Euclidean spaces 
the authors of \cite{NemethNemeth2012} showed that $P_K$ is isotone with respect to the order 
induced by $K$ if and only if $P_L$ is subadditive with respect to the order induced by $L$, 
where $K$ and $L$ are mutually dual pointed closed convex cones. If $K$ is also pointed and 
generating, then the isotonicity of $P_K$ with 
respect to the order induced by $K$ implies the latticiality of the cone in Hilbert spaces as 
well (see \cite{IsacNemeth1990,IsacNemeth1990b,IsacNemeth1990c}). The main result of this paper
states that $P_K$ is isotone with respect to the order induced by $K$ (i.e., $v-u\in K$ implies $P_Kv-P_Ku\in K$) if and only if $P_L$ is 
subadditive with respect to the order induced by $L$
(i.e., $P_Lu+P_Lv-P_L(u+v)\in L$ for any $u,\;v \in \R^n$), where $K$ and $L$ are mutually dual 
pointed closed convex cones of a Hilbert space, thus extending the result of 
\cite{NemethNemeth2012}. This result also implies that if $K$ is a pointed generating cone 
in a Hilbert space such that $P_K$ is subadditive with respect to the order induced by 
$K$, then it must be latticial. The latter two results have been already proved in Euclidean 
spaces (see \cite{NemethNemeth2012} and \cite{IsacNemeth1992}), but they were open until now 
in Hilbert spaces, except for the particular case of a Hilbert lattice \cite{Nemeth2003}. 
Although originally 
motivated by complementarity problems, recently it turned out that the isotonicity and 
subadditivity of projections are also motivated by other practical problems at least as
important as the complementarity problems such as the problem of map-making from relative 
distance information e.g., stellar cartography  
(see 
\vspace{2mm}

\noindent {\small www.convexoptimization.com/wikimization/index.php/Projection\_on\_Polyhedral\_Convex\_Cone}
\vspace{2mm}

\noindent and Section 5.13.2 in \cite{Dattorro2005}) and isotone regression 
\cite{GuyaderJegouNemeth2012}, where the equivalence between two classical algorithms in 
statistics is proved by using theoretical results about isotone projections. We remark that 
our proofs are essentially infinite dimensional and apparently there is no easy way to extend 
the methods of \cite{NemethNemeth2012} to infinite dimensions. The paper 
\cite{GuyaderJegouNemeth2012} shows that investigation of the structure of cones admitting 
isotone and subadditive projections is important for possible future applications. The proofs 
presented here also provide a more elegant way of 
proving the results of \cite{NemethNemeth2012}. However, the difference is that they do not
contain the proof of the latticiality of the involved cones. (For pointed generating cones in 
Hilbert spaces this is the consequence of the main result in \cite{IsacNemeth1990}.)

The structure of this note is as follows: After
some preliminary terminology we introduce the main tools for our proofs the Moreau's 
decomposition theorem (i.e., Lemma \ref{lm}) and the lattice-like
operations related to a projection onto a cone and then we proceed to showing our main result. 

\section{Preliminaries}

Let $\Hi$ be a a real Hilbert space endowed with a scalar product $\lng\cdot,\cdot\rng$ and let 
$\|.\|$ the norm generated by the scalar product $\lng\cdot,\cdot\rng$.

Throughout this note we shall use some standard terms and results from convex geometry 
(see e.g. \cite{Rockafellar1970}). 

Let $K$ be a \emph{closed convex cone} in $\Hi$, i. e., a nonempty closed set with 
$tK+sK\subset K,\;\forall \;t,s\in \R_+=[0,+\infty)$. The closed convex cone $K$ is called 
\emph{pointed}, if $K\cap(-K)=\{0\}.$

The cone $K$ is {\it generating} if $K-K=\Hi$.
 
The convex cone $K$ defines a pre-order relation (i.e., a reflexive and transitive binary 
relation) $\leq_K$, where $x\leq_Ky$ if and only if $y-x\in K$. 
The relation is {\it compatible with the vector structure} of $\Hi$ in the sense that 
$x\leq_K y$ implies $tx+z\leq_K ty+z$ for all $z\in \Hi$, and all 
$t\in \R_+$. If $\sqsubseteq$ is a reflexive and transitive relation on $\Hi$ which is 
compatible with the vector structure of $\Hi$, then $\sqsubseteq=\leq_K$ with 
$K=\{x\in\Hi:0\sqsubseteq x\}.$ If $K$ is pointed, then $\leq_K$ it is \emph{antisymmetric} 
too, that is $x\leq_K y$ and 
$y\leq_K x$ imply that $x=y.$ Hence, in this case $\le_K$ becomes an order relation (i.e, a
reflexive, transitive and antisymmetric binary relation). The elements $x$ and $y$ are called 
\emph{comparable} if $x\leq_K y$ or $y\leq_K x.$

We say that $\leq_K$ is a \emph{latticial order} if for each pair of elements $x,y\in \Hi$ 
there exist the lowest upper bound $\sup\{x,y\}$ (denoted by $x\vee y$) and the uppest lower bound $\inf\{x,y\}$ of
the set $\{x,y\}$ (denoted by $x\wedge y$)  with respect to the relation $\leq_K$. In this case $K$ is said a 
\emph{latticial or simplicial cone}, and $\Hi$ equipped with a latticial order is called a
\emph{Riesz space} or \emph{vector lattice}.

The \emph{dual} of the convex cone $K$ is the set 
$$K^*:=\{y\in \Hi:\;\lng x,y\rng \geq 0,\;\forall x\in K\}.$$ 
The set $K^*$ is  a closed convex cone.
If $K$ is a closed cone, then the extended Farkas lemma (see Exercise 2.31 (f)
in \cite{BoydVandenberghe2004}) says that $(K^*)^*=K.$
Hence denoting $L=K^*$ we see that $K=L^*$ and $L^*=K$. For the
closed cones $K$ and $L$ related by these relations we say that they
are \emph{mutually dual cones}.
The cone $K$ is called \emph{self-dual}, if $K=K^*.$ If $K$ is self-dual, then it is a 
generating, pointed, closed convex cone.

Let $K$ be a closed convex cone and $\rho:\Hi\to\Hi$ a mapping. Then, $\rho$ is called \emph{$K$-isotone} if $x\le_K y$ implies $\rho(x)\le_K\rho(y)$ and \emph{$K$-subadditive} 
if $\rho(x+y)\le_K\rho(x)+\rho(y)$, for any $x,y\in H$. 

Denote by $P_D$ the projection mapping onto a nonempty closed  convex set $D$ of the Hilbert 
space $\Hi,$ that is the mapping which associates to $x\in \Hi$ the unique nearest point of 
$x$ in $D$ (\cite{Zarantonello1971}):

\[ P_Dx\in D,\;\; \textrm{and}\;\; \|x-P_Dx\|= \inf \{\|x-y\|: \;y\in D\}. \]

Next, we shall frequently use the 
following simplified form of the Moreau's decomposition theorem \cite{Moreau1962}:
\begin{lemma}\label{lm}
	Let $K$ and $L$ be mutually dual cones in the Hilbert space $\Hi$. For any $x$ in $K$ 
	we have $x=P_Kx-P_L(-x)$ and $\lng P_Kx,P_L(-x)\rng=0$. The relation $P_Kx=0$ holds if 
	and only if $x\in -L$.
\end{lemma}

Let $K$ and $L$ be mutually dual cones in the Hilbert space $\Hi$. Define the 
following operations in $\Hi$: 
\[x\sa_K y=P_{x-K}y,\textrm{ }x\su_K y=P_{x+K}y,\textrm{ }x\sa_L y=P_{x-L}y,\textrm{ and }
x\su_L y=P_{x+L}y.\]
Assume that
the operations $\su_K$, $\sa_K$, $\su_L$ and $\sa_L$ have precedence over the addition of 
vectors and multiplication of vectors by scalars.

If $K$ is self-dual, then $\su_K = \su_L$ and $\sa_K=\sa_L$ and we arrive to the generalized
lattice operations defined by Gowda, Sznajder and Tao in \cite{GowdaSznajderTao2004}, and used 
by our paper~\cite{NemethNemeth2012b}. 

A direct checking yields that if $K$ is a self-dual latticial cone, then 
$\sa_K=\sa_L=\wedge$, and $\su_K=\su_L=\vee$.
That is $\sa_K$, $\sa_L$, $\su_K$ and $\su_L$ are some \emph{lattice-like operations}.

We shall simply call a set $M$ which is invariant with respect to the operations $\sa_K$, 
$\sa_L$, $\su_K$ and $\su_L$ \emph{$K$-invariant}. 

The following Theorem greatly extends Lemma 2.4 of \cite{NishimuraOk2012} and since it can be 
shown very similarly to Theorem 1 of \cite{NemethNemeth2012c} (i.e., the corresponding result
in Euclidean spaces), we state it here without proof.

\begin{theorem}\label{ISO}
	Let $K\subset\Hi$ be a closed convex cone and $C\subset\Hi$ be a closed convex set. Then $C$ is $K$-invariant, if and 
	only if $P_C$ is $K$-isotone. 
\end{theorem}

\section{The main result}

\begin{theorem}\label{tm}
	Let $K,L$ be mutually dual closed convex cones in a Hilbert space $\Hi$. Then, the 
	following two statements are equivalent:
	\begin{enumerate}[(i)]
		\item\label{tma} $P_K$ is $K$-isotone.
		\item\label{tmb} $P_L$ is $L$-subadditive.
	\end{enumerate}
\end{theorem}

\begin{proof}
	\vspace{1mm}

	(\ref{tma})$\implies$(\ref{tmb}): From Theorem \ref{ISO}, it follows that $K$ is $K$-invariant. From the definition of the $K$-invariance it is easy to see that $K$ is also 
	$L$-invariant. Hence, by using again Theorem \ref{ISO}, it follows that $P_K$ is $L$-isotone. Now let $x,y\in\Hi$ be arbitrary. From Lemma \ref{lm}, we have
	$x+y\le_L P_K(x)+P_K(y)$, because \[P_K(x)+P_K(y)-x-y=P_K(x)-x+P_K(y)-y=P_L(-x)+P_L(-y)\in L+L\subset L.\] Hence, by the $L$-isotonicity of $P_K$, we have 
	\[P_K(x+y)\le_L P_K(P_K(x)+P_K(y))=P_K(x)+P_K(y),\] which means that $P_K$ is $L$-subadditive. Thus, by using Lemma \ref{lm}, we get 
	\begin{gather*}
		P_L(x+y)=x+y+P_K(-x-y)\le_L x+y+P_K(-x)+P_K(-y)\\=x+P_K(-x)+y+P_K(-y)
		=P_L(x)+P_L(y),
	\end{gather*} 
	which is equivalent to the $L$-subadditivity of $P_L$. 
	\vspace{1mm}

	(\ref{tmb})$\implies$(\ref{tma}):  Let $x\le_L y$. Then, $y-x\in L$ and therefore, by the $L$-subadditivity of $P_L$ and Lemma \ref{lm}, we get 
	\begin{eqnarray*}
		P_K(x)=P_L(-x)+x=P_L(y-x-y)+x\le_L P_L(y-x)+P_L(-y)+x\\=y-x+P_L(-y)+x=P_K(y)
	\end{eqnarray*}
	Hence, $P_K$ is $L$-isotone and therefore Theorem \ref{ISO} implies that $K$ is $L$-invariant. From the definition of the $K$-invariance it is easy to see that $K$ is also 
	$K$-invariant. Therefore, by using again Theorem \ref{ISO}, it follows that $P_K$ is $K$-isotone.
\end{proof}



\bibliographystyle{habbrv}
\bibliography{arxiv-hildual}

\end{document}